\documentclass[10pt,a4paper]{amsart}
\usepackage{amsthm,amsmath,amssymb}

\usepackage[ margin=1.5in]{geometry}
     \usepackage{amssymb}
	\usepackage{graphicx}
           \graphicspath{{c:/users/ali/pictures/}}

\usepackage{bm}
\usepackage{subfig,tikz}
\usepackage{wrapfig}
\usepackage{color}
\usepackage{xspace}
\usetikzlibrary{shapes.geometric}
\usepackage{url}

      \theoremstyle{plain}
      \newtheorem{theorem}{Theorem}[section]
      \newtheorem{lemma}[theorem]{Lemma}
      \newtheorem{corollary}[theorem]{Corollary}
      \theoremstyle{Conjecture}
       
      \theoremstyle{definition}
      \newtheorem{definition}[theorem]{Definition}
      
      \theoremstyle{remark}
      \newtheorem{remark}[theorem]{Remark}

      \makeatletter
      \def\@setcopyright{}
      \def\serieslogo@{}
      \makeatother
   
\usepackage[english]{babel}

\DeclareMathOperator{\spn}{span}

\begin{document}

   \author {{Ali Feizmohammadi}}
   \address{University College London}
   \email{a.feizmohammadi@ucl.ac.uk}

   % second author

 %  \author{Stephen G.~Simpson}

   % the address where the research was carried out
   %\address{University of G\"ottingen, G\"ottingen, Germany}

   % current address, usually not needed because it is the same as the
   % regular address
   %\curraddr{Department of Mathematics, University of Toronto,
     %University Park, State College PA 16802}

   %\email{simpson@math.psu.edu}
   
   % title

%\tableofcontents
%\newpage

   \title[Local DN map and locally CTA geometries]{Uniqueness of a Potential From Local Boundary Measurements }

   % Note that the short title for running heads goes in square
   % brackets.  This is optional.  The long title goes in curly
   % braces.  In the long title, line breaks are indicated by \\.

   % abstract (optional)
   \begin{abstract}
   Let $(\Omega^3,g)$ be a compact smooth Riemannian manifold with smooth boundary and suppose that $U$ is an open set in $\Omega$ such that $g|_U$ is the Euclidean metric. Let $\Gamma= \overline{U} \cap \partial \Omega$ be non-empty, connected, strictly convex and assume that $U$ is the convex hull of $\Gamma$. We will study the uniqueness of an unknown potential for the Schr\"{o}dinger operator $ -\triangle_g +  q  $ from the associated local Dirichlet to Neumann map, $C_q^{\Gamma,\Gamma}$. Indeed, we will prove that if the potential $q$ is a priori explicitly known in $U^c$, then one can uniquely reconstruct $q$ from the knowledge of $C^{\Gamma,\Gamma}_q$. 
   \end{abstract}

   % AMS subject classifications (used in AMS journals)
   %\subjclass{Primary 00A30; Secondary 00A22, 03E20}

   % AMS keywords (used in AMS journals)
   %\keywords{infinite, seminar}

   % acknowledge support, etc
  % \thanks{This research was partially supported by NSF grant
     %DOA-123456789.}
   %\thanks{We would like to thank our colleagues for their helpful
    % criticism.}

   % dedication
   %\dedicatory{Dedicated to Professor Donald Knuth on the occasion
     %of his $100$th birthday}

   % today's date, or fill in whatever date you prefer
  % \date{\today}

% This ends the top matter information.
% We can now tell LaTeX to display the top matter.

   \maketitle

% Having displayed the top matter, we now proceed to the body of the
% article.

% The body of the article is divided into sections.
% Each section begins with a \section command.

\tableofcontents

\section{Problem Formulation}
 %Suppose that $(\Omega^3,g)$ is a compact smooth three dimensional Riemannian manifold. Let us define the Dirichlet to Neumann map $\Lambda_{q}: H^{\frac{1}{2}}(\partial \Omega) \to  H^{-\frac{1}{2}}(\partial \Omega)$ for Schr\"{o}dinger operator as follows:
%\[ (\Lambda_{q} f, h) = \int_{\Omega} \langle du,dv\rangle_g dV_g + \int_{\Omega} quv dV_g \]
%where $ v \in H^1(\Omega)$ satisfies $ v|_{\partial\Omega} = h$ and $ u \in H^1(\Omega)$ satisfies $ (-\triangle_g+q) u=0$ with $ u|_{\partial\Omega}=f$. Let us impose the condition that $0$ is not a Dirichlet eigenvalue for $-\triangle_g + q$.  $\Lambda_q$ is a self adjoint bounded linear Elliptic Pseudo Differential Operator [4]. For smoother Dirichlet data $f \in H^{\frac{3}{2}}(\partial \Omega)$ we have $ \Lambda_q f = \partial_{\nu} u |_{\partial\Omega}$ where $\nu$ is the outward normal unit vector field on the boundary of the domain.\\
 Let $(\Omega^3,g)$ be a three dimesnional compact smooth Riemannian manifold with smooth boundary. Let $\Gamma_D$, $\Gamma_N$ denote open subsets of the boundary $\partial \Omega$ and assume that $q \in C(\overline{\Omega})$ is an unknown function. Consider the Cauchy data:
$$ C_q^{\Gamma_D,\Gamma_N} = \{(u|_{\Gamma_D},\partial_{\nu}u|_{\Gamma_N}): (-\triangle_g + q)u=0 \hspace{2mm} \text{on} \hspace{2mm} \Omega, u \in H_{\triangle}(\Omega), supp(u|_{\partial \Omega})=\Gamma_D \}$$

\noindent Here, we are using the Hilbert space $H_{\triangle}(\Omega)=\{ u \in L^2(\Omega), \triangle_g u \in L^2(\Omega)$\}.  The trace $u|_{\partial \Omega}$ and the normal derivative $\partial_{\nu} u|_{\partial \Omega}$ belong to $H^{-\frac{1}{2}}(\partial \Omega)$ and $H^{-\frac{3}{2}}(\partial \Omega)$ respectively. Moreover, if $u \in H_{\triangle}(\Omega)$ and $Tr(u) \in H^{\frac{3}{2}}(\partial \Omega)$ then $ u \in H^2(\Omega)$. For more details about the space $H_{\triangle}$ we refer the reader to \cite{BU} and \cite{NS}.  

\noindent The partial data version of the Calder\'{o}n conjecture asks whether the knowledge of $C_q^{\Gamma_D,\Gamma_N}$ uniquely determines $q$?\\ 
\noindent Let us first discuss the main results in the literature. When $g$ is Euclidean and $\Gamma_D=\Gamma_N=\partial \Omega$ the uniqueness of the potential was proved in \cite{SU}. In the case where $g$ is Euclidean but $\Gamma_D, \Gamma_N$ may not be the whole boundary there also exists several uniqueness results (the following are extracted from \cite{KS}):

\begin{itemize}

\item{ The set $\Gamma_D$ is very small and $\Gamma_N$ contains $\partial \Omega \setminus \Gamma_D$. The uniqueness of the potential function in $\Omega$ is proved in \cite{KSU}.}

\item{$\Gamma_D=\Gamma_N$. The uniqueness of the potential function in the convex hull of $\Gamma_D$ is proved in \cite{GU}.}

\item{ $\Gamma_D=\Gamma_N$ is either part of a hyperplane or part of a sphere. The uniqueness of the potential function in $\Omega$ is proved in \cite{I}.}

\item{The linearized Calder\'{o}n problem, $\Gamma_D=\Gamma_N$ can be an arbitrary open subset of $\partial \Omega$. The uniqueness of the potential function is proved in \cite{DKSJU}.} 

\end{itemize}

\noindent When the metric is not flat, the most general result exists for conformally transversally anisotropic geometries (CTA) \cite{DKSU}. These are geometries where the manifold has a product structure $\Omega=\mathbb{R} \times \Omega_0$ with $(\Omega_0,g_0)$ being called the transversal manifold and the metric takes the form:
\[ g = c(t,x) (dt^2 + g_0(x)).\]

\noindent When $\Gamma_D=\Gamma_N=\partial \Omega$, uniqueness of the potential function was proved in \cite{DKSU} under a geometric assumption on the transversal manifold $\Omega_0$. Subsequently, in \cite{DKLS} the result was improved by requiring weaker restrictions on $\Omega_0$. In the case when $\Gamma_D, \Gamma_N$ may not be the whole boundary, local uniqueness of the potential function in CTA geometries was proved in \cite{KS}. Global uniqueness was also deduced in the same paper under an additional concavity assumption.

\noindent In this paper we will prove the following theorem:

\begin{theorem}
\label{partial}
Let $(\Omega^3,g)$ denote a compact smooth Riemannian manifold with smooth boundary. Let $U \subset \Omega$ be an open subset such that $\Gamma=\overline{U} \cap \partial \Omega$ is non-empty, connected and strictly convex. Suppose $U$ can be covered with a coordinate chart in which $g|_U$ is the Euclidean metric and that $U$ is the convex hull of $\Gamma$.   Suppose $q \in C(\overline{\Omega})$ is an unknown function and suppose $q-q_*$ is compactly supported in $U$ where $q_* \in C(\overline{\Omega})$ is an explicitly known continuous function. Then the knowledge of $C_q^{\Gamma,\Gamma}$ will uniquely determine $q$ on $U$.\\ 
\end{theorem}

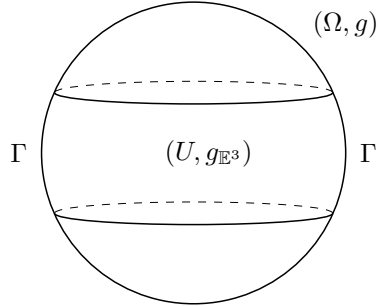
\begin{figure}[h]
           \centering
	\begin{tikzpicture}
	\shade[ball color = white!, opacity = 0] (0,0) circle (2cm);
	\draw[line width = 0.2mm](0,0) circle (2cm);
	\draw[line width = 0.2mm] (-1.83,0.8) arc (180:360:1.83 and 0.15);
	\draw[line width = 0.1mm, dashed] (1.83,0.8) arc (0:180:1.83 and 0.15);
	\node[scale =1] at (0.2,0) {$(U,g_{\mathbb{E}^3})$};
	\node[scale =1] at (2,1.7) {$(\Omega,g)$};
	\node[scale =1] at (2.3,0) {$\Gamma$};
	\node[scale =1] at (-2.3,0) {$\Gamma$};
	\draw[line width = 0.2mm] (-1.83,-0.8) arc (180:360:1.83 and 0.15);
	\draw[line width = 0.1mm, dashed] (1.83,-0.8) arc (0:180:1.83 and 0.15);
	\end{tikzpicture}
\end{figure}

\begin{remark}
We would like to point out that in \cite{Feiz} we studied the problem of uniqueness of the potential function for the same geometries and in the case where $\Gamma_D=\Gamma_N=\partial \Omega$ and derived the uniqueness result. The method presented in that paper is quite robust and in fact it can be adjusted to yield Theorem ~\ref{partial}. Indeed, the key difference in this paper is the modification to the Carleman estimate in Lemma 5.4  \cite{Feiz}.
\end{remark}

\section{Carleman Estimates}

\noindent We begin by setting up local coordinate systems in our manifold $(\Omega,g)$ which will be useful for the construction of several key functions in the proof. Note that $U$ has a foliation by a family of planes $\mathbb{A}=\{\Pi_t\}_{t \in I}$ where $I = [0,1]$. We start by taking a fixed plane $\Pi \in \mathbb{A}$. A local coordinate system $(x_1,x_2,x_3)$ can be constructed in $U$ such that $\Pi=\{x_3=0\}$ with $(x_1,x_2)$ denoting the usual cartesian coordinate system on the plane $\Pi$ and $\partial_3$ denoting the normal flow to this plane. We can assume that the support of $q-q_*$ lies in the compact set $V \subset\subset \{-t_1\leq x_3 \leq t_2\}$ with $t_1,t_2>0$. In this framework $U= \cup_{c=-t_1-2\delta_1}^{c=t_2+2\delta_2} \{x_3=c\}$ with $\delta_i >0$ for $i \in \{1,2\}$.\\

\noindent Throughout the paper we will use the Fermi coordinates near a surface. Let us recall the construction of Fermi Coordinates in a Riemannian manifold $(M^3,g)$ near a non-degenerate orientable surface $\Sigma$. We will follow \cite{pacard} here.\\
\noindent  Let  $N$ denote the normal unit vector field on $\Sigma$ which defines the orientation of $\Sigma$. We make use of the exponential map to define:
$$ Z(y, z) := Exp_y(z N(y))$$
Here $y \in \Sigma$ and $z \in \mathbb{R}$. The implicit function theorem implies that $Z$ is a local diffeomorphism from a neighborhood of a point $(y, 0) \in \Sigma \times \mathbb{R}$ onto a neighborhood of $ y \in M$. For any $z \in \mathbb{R}$ we define $\Sigma_z = \{Z(y,z) \in M: y \in \Sigma\}$. Let $g_z$ denote the induced metric on $\Sigma_z$. Gauss's Lemma implies that:
$$ Z^* g = dz^2 + g_z$$ 
Here, $g_z$ is considered as a family of metrics on $T\Sigma$ smoothly depending on the variable $z$. In fact we have the following Taylor series expansion near $\{z=0\}$:
$$ g_z = g_0 - 2zh_0 + O(z^2)$$
Here, $g_0$ and $h_0$ denote the induced metric and the second fundamental form on $\Sigma$ respectively.\\

\noindent With this review of the Fermi coordinates, let us proceed with the construction of the local coordinates in our manifold $(\Omega,g)$.

\noindent We will denote the region outside of $V$ and above $\Pi$ by $W_u$ and the other remaining region outside of $V$ and below $\Pi$ by $W_l$. Let us consider the two surfaces $S_u = (\partial \Omega \setminus \overline{U}) \cap W_u$ and $ S_l = (\partial \Omega \setminus \overline{U}) \cap W_l$. Let $(z_1,z_2,z_3)$ and $(s_1,s_2,s_3)$ denote the Fermi coordinates near the two surfaces $S_l$ and $S_u$ respectively.  Recall that in these local coordinates we have that:
$$ Z^*g = dz_3^2 + g_{z_3}$$
near the surface $S_l$ and:
$$Z^*g =ds_3^2+ g_{s_3}$$
near the surface $S_u$.

\begin{definition}
Let us define two smooth functions $\omega:\Omega \to \mathbb{R}$ and $\tilde{\omega}:\Omega \to \mathbb{R}$ as follows:
\begin{itemize}

\item{Let $\omega:\Omega \to \mathbb{R}$ be any smooth function such that $d\omega \neq 0$ everywhere in $\Omega$, $\omega(x) \equiv x_3$  for $ -t_1-\delta_1 \leq x_3 \leq t_2+\delta_2$, $\omega \equiv s_1$ near $S_u$ and $\omega \equiv z_1$ near $S_l$.}
\item{Let $\tilde{\omega}:\Omega \to \mathbb{R}$ be any smooth function such that $\tilde{\omega}(x) \equiv x_2$  for $ x \in U$.}
\end{itemize}
\end{definition} 

\noindent Clearly existence of such a function as $\tilde{\omega}$ is trivial. The existence of such a function as $\omega$ will be the content of Lemma ~\ref{nontrapping}:

\begin{lemma}
\label{nontrapping}
There exists a function $\omega: \Omega \to \mathbb{R}$ satisfying the above properties.
\end{lemma}

\begin{proof}

Recall that Morse Lemma states the following: 
Let $b$ be a non-degenerate critical point of $f: \Omega \to \mathbb{R}$. Then there exists a chart $(p_1,p_2,p_3)$ in a neighborhood of $b$ such that $$f(p)= f(b)-p_1^2-...-p_{\alpha}^2 + p_{\alpha+1}^2+...+p_3^2$$ Here $\alpha$ is equal to the index of $f$ at $b$.\\
Define $\omega_0: \Omega \to \mathbb{R}$ such that $\omega_0(x) \equiv x_3$  for $ -t_1-\delta_1 \leq x_3 \leq t_2+\delta_2$, $\omega_0 \equiv s_1$ near $S_u$ and $\omega_0 \equiv z_1$ near $S_l$ . If $d\omega_0 \neq 0$ anywhere then the proof is complete so let us suppose that $\omega_0$ has critical points. We know that a generic smooth function is Morse and therefore it has isolated critical points. By using a small $C^{\infty}$ purturbation we can find a smooth function $\omega_1(x)$ such that $\omega_1(x) \equiv x_3$  for $ -t_1-\delta_1 \leq x_3 \leq t_2+\delta_2$, $\omega_1 \equiv s_1$ near $S_u$, $\omega_1 \equiv z_1$ near $S_l$ and such that $\omega_1(x)$ has isolated critical points and thus by compactness a finite number of isolated critical points $b_k$ for $1 \leq k \leq L$. We will assume without loss of geneality that the index of these critical points is zero.\\ Since $\dim \Omega=3>2$, we can connect these critical points with points just outside the boundary by a family of disjoint paths that do not intersect $V$ or $S_l$ or $S_u$. The idea here is to remove these critical points from $\Omega$ by pushing them out of the manifold. We will denote these curves by $\gamma_k$. Let $V_k$ denote the neighborhood around $b_k$ for which the Morse lemma holds. Choose $h$ small enough such that the geodesic ball of radius $h$ around $b_k$ is inside $V_k$ namely $B_{b_k}(h) \subset V_k$. Take $$\omega_2(x)= \omega_1(x) + \epsilon (\alpha p_1 + \beta p_2 +\lambda p_3) \eta_k(x)$$ where $\eta_k$ is a smooth function compactly supported in $V_k$ and such that $\eta_k \equiv 1 $ in the ball $B_{b_k}(\frac{h}{2})$. It is clear that for $\epsilon$ small enough we still have that $\omega_2(x) \equiv x_3$  for $ -t_1-\delta_1 \leq x_3 \leq t_2+\delta_2$, $\omega_2 \equiv s_1$ near $S_u$ and $\omega_2 \equiv z_1$ near $S_l$ . Furthermore we can see that for $\epsilon$ small enough the critical points of $\omega_2$ outside $V_k$ will remain the same and the critical point of $\omega_2$ inside $V_k$ must be in the ball $B_{b_k}(\frac{h}{2})$. Hence the critical point in $V_k$ will 'move' from $b_k$ to the point with local coordinate $(p_1,p_2,p_3) = (\frac{\epsilon \alpha}{2},\frac{\epsilon \beta}{2},\frac{\epsilon \lambda}{2})$. Since $\Omega$ is compact, it is clear that we can move the critcal points $b_k$ along their respective curves $\gamma_k$ and essentially construct a smooth function $\omega$ satisfying the desired properties. \\
\end{proof}
\bigskip

%\noindent From this point onwards we let:$$W:= \partial \Omega \cap \overline{U}.$$

\begin{definition} 
\label{D}
$$\mathbb{D} := \{ v \in C^2(\Omega) : v|_{\partial \Omega}=0, \partial_{\nu} v|_{\Gamma} =0\} $$
\end{definition}

\begin{definition}
Let us define two globally defined $C^{k-1}(\overline{\Omega})$ functions $\chi_0 : \Omega \to \mathbb{R}$ and $F_{\lambda}: \mathbb{R} \to \mathbb{R}$  as follows:
 \[
    \chi_0(x) = \left\{\begin{array}{lr}
        1, & \text{for }  -t_1<x_3<t_2\\
        (1-(\frac{x_3-t_2}{\delta_2})^{8k})^k, & \text{for } t_2 \leq x_3 \leq t_2+\delta_2\\
    (1-(\frac{x_3+t_1}{\delta_1})^{8k})^k, & \text{for } -t_1-\delta_1 \leq x_3 \leq -t_1\\
         0  & \text{otherwise }
        \end{array}\right\}
  \]

 \[
    F_{\lambda} (x) = \left\{\begin{array}{lr}
        0, & \text{for }   -t_1<x<t_2\\
        e^{\lambda (\frac{x-t_2}{\delta_2})^2} (\frac{x-t_2}{\delta_2})^{2k}, & \text{for } t_2 \leq x  \\
       e^{\lambda (\frac{x+t_1}{\delta_1})^2} (\frac{x+t_1}{\delta_1})^{2k}  , & \text{for }  x \leq -t_1\\
        \end{array}\right\}\\
\\
  \]
\end{definition}

\bigskip

%Let $\alpha <0$ be sufficiently large and define:
%$$ \tilde{x} = e^{\alpha x}$$
%$$ \tilde{t}_1 = e^{\alpha t_1}$$
%$$\tilde{t}_2=e^{-\alpha t_2}$$
%$$\tilde{\delta_2}=e^{\alpha(t_2+\delta_2)}-e^{\alpha t_2}$$
%$$\tilde{\delta_2}=-e^{\alpha(-t_1-\delta_1)}+e^{-\alpha t_1}$$

%\noindent Let us define two globally defined $C^{k}(\overline{\Omega_1})$ functions $\chi_0 : \Omega_1 \to \mathbb{R}$ and $F_{\lambda}(x): \mathbb{R} \to \mathbb{R}$  as follows:
 %\[
   % \chi_0(x) = \left\{\begin{array}{lr}
     %   1, & \text{for }  -t_1<x_3<t_2\\
       % (1-(\frac{\tilde{x}_3-\tilde{t}_2}{\tilde{\delta}_2})^{8k})^k, & \text{for } t_2 \leq x_3 \leq t_2+\delta_2\\
    %(1-(\frac{\tilde{x}_3-\tilde{t}_1}{\tilde{\delta}_1})^{8k})^k, & \text{for } -t_1-\delta_1 \leq x_3 \leq -t_1\\
       %  0  & \text{otherwise }
        %\end{array}\right\}
 % \]

 %\[
 %   F_{\lambda} (x) = \left\{\begin{array}{lr}
    %    0, & \text{for }   -t_1<x<t_2\\
      % e^{\lambda (\frac{\tilde{x}-\tilde{t}_2}{\tilde{\delta}_2})^2} (\frac{\tilde{x}-\tilde{t}_2}{\tilde{\delta}_2})^{2k}, & \text{for } t_2 \leq x  \\
     %  e^{\lambda (\frac{\tilde{x}-\tilde{t}_1}{\tilde{\delta}_1})^2} (\frac{\tilde{x}-\tilde{t_1}}{\tilde{\delta}_1})^{2k}  , & \text{for }  x \leq -t_1\\
       % \end{array}\right\}\\
%\\
  %\]

\noindent Using the explicit functions above, we can proceed with the following key lemma that will help us obtain a Carleman estimate in $(\Omega,g)$. The proof of Lemma ~\ref{hormander} will closely follow the proof presented in \cite{Feiz}.

\begin{lemma}
\label{hormander}

Let $ \phi_0(x_1,x_2,x_3) = x_1 \chi_0(x) + (F_{\lambda}\circ \omega)(x)$ where  $k \geq 1$ is an arbitraty integer and $\lambda(\Omega, k,||g_{ij}||_{C^2})$ is sufficiently large. Then the H\"{o}rmander hypo-ellipticity condition is satisfied in $\Omega$, that is to say:
\[ D^2\phi_0 (X,X) + D^2\phi_0(\nabla \phi_0,\nabla\phi_0) \geq 0\]
whenever $|X|=|\nabla \phi_0|$ and $\langle \nabla\phi_0, X \rangle =0$.\\
\end{lemma}

\begin{proof}[Proof of Lemma ~\ref{hormander}]

The proof will be divided into three parts. We will consider the the three regions $A_1=\{ -t_1 \leq x_3 \leq t_2 \}$ , $ A_2=\{  t_2 \leq  x_3 \leq t_2+\delta_2 \} \cup \{-t_1-\delta_1 \leq x_3 \leq -t_1\}$ and $A_3 = \Omega \setminus (A_1 \cup A_2)$ and prove the inequality holds in all these regions. Recall that the  metric is Euclidean on $U$ which implies that both $A_1$ and $A_2$ are Euclidean. Let us first consider $A_1$. Note that in this region $ \phi_0(x_1,x_2,x_3) = x_1$ and since the metric is Euclidean in this region we deduce that $ D^2  \phi_0 (X,Y) \equiv 0$ for all $X,Y$ and hence the H\"{o}rmander condition is satisfied.\\

\noindent Let us now focus on the region denoted by $A_3$. Notice that in this region we have $ \phi_0 = F_{\lambda}(\omega(x))$. Therefore the level sets of $\phi_0(x)$ will simply be the level sets $ \{\omega(x) =c\}$. 

$$  D^2\phi_0(X,X) = \langle D_X \nabla \phi_0,X \rangle$$

\noindent Since $ |X| =|F'_{\lambda}(\omega)||\nabla \omega|$ we obtain the following estimate: 
$$  D^2\phi_0(X,X) \leq C |F_{\lambda}'(\omega)|^3$$

\noindent where it is important to note that the constant $C$ is independent of $\lambda$. Furthermore we have:

$$ D^2\phi_0(\nabla \phi_0,\nabla \phi_0) = \frac{1}{2} \nabla \phi_0 (|\nabla \phi_0|^2). $$
Since $\phi_0 = F_\lambda ( \omega(x))$:
$$  D^2\phi_0(\nabla \phi_0,\nabla \phi_0)= \frac{1}{2}( F'(\omega)^3 \nabla \omega( |\nabla \omega|^2) + 2F'(\omega)^2F''(\omega) |\nabla \omega|^4).  $$

\noindent One can easily check that for $x \in A_3$:

 \[
    |F'_{\lambda} (\omega)| \leq \left\{\begin{array}{lr}
        C \lambda e^{ \lambda (\frac{\omega-t_2}{\delta_2})^2}& \text{for } t_2+\delta_2 \leq x_3 \\
        C \lambda e^{  \lambda (\frac{\omega+t_1}{\delta_1})^2}, & \text{for } x_3 \leq -t_1-\delta_1\\
        \end{array}\right\}\\
\\
  \]

 \[
    F''_{\lambda} (\omega) \geq \left\{\begin{array}{lr}
        C \lambda^2 e^{ \lambda (\frac{\omega-t_2}{\delta_2})^2}& \text{for } t_2+\delta_2 \leq x_3 \\
        C \lambda^2 e^{  \lambda (\frac{\omega+t_1}{\delta_1})^2}, & \text{for } x_3 \leq -t_1-\delta_1\\
        \end{array}\right\}\\
\\
  \]

\noindent Thus we can easily conclude that for $\lambda$ large enough the H\"{o}rmander hypoellipticity condition is satisfied in this region. Let us now turn our attention to the  transition region $x \in A_2$. Recall that the metric $g$ is flat in $A_2$. We will actually prove the stronger claims:
\\
(1) $ D^2\phi_0(\nabla \phi_0,\nabla \phi_0) \geq 0, $
\\
(2) $ D^2\phi_0 (X,X) \geq 0 $ for all $X$ with $ \langle \nabla \phi_0, X \rangle =0.$
\\
 The idea is that near the $\{x_3=0\}$ hypersurface the convexity of  $x_3^{2k} e^{\lambda x_3^2}$ yields the H\"{o}rmander Hypo Ellipticity. Furthermore away from this surface a suitable choice of $\lambda$ large enough will yield non-negativity as well thus completing the proof. We will now make these statements more precise as follows:

 \[
    F'_{\lambda} (x) = \left\{\begin{array}{lr}
     (\frac{x_3-t_2}{\delta_2})^{2k-1}e^{\lambda(\frac{x_3-t_2}{\delta_2})^2}(\frac{2k+2\lambda(\frac{x_3-t_2}{\delta_2})^2}{\delta_2}) , & \text{for } t_2 \leq x_3 \leq t_2+\delta_2 \\
      (\frac{x_3+t_1}{\delta_1})^{2k-1}e^{\lambda(\frac{x_3+t_1}{\delta_1})^2}(\frac{2k+2\lambda(\frac{x_3+t_1}{\delta_1})^2}{\delta_1}) , & \text{for } -t_1-\delta_1 \leq x_3 \leq -t_1 \\
        \end{array}\right\}\\
\\
  \]

 %\noindent Similarly:

 $$ F''_{\lambda} (x) =  (\frac{x_3-t_2}{\delta_2})^{2k-2} e^{ \lambda(\frac{x_3-t_2}{\delta_2})^2}((\frac{(2k)(2k-1)}{\delta_2^2}+\frac{8\lambda k +2\lambda}{\delta_2^2}(\frac{x_3-t_2}{\delta_2})^2+ \frac{4\lambda^2}{\delta_2^2}(\frac{x_3-t_2}{\delta_2})^4 )$$ for  $t_2 \leq x_3 \leq t_2+\delta_2$ and: \\
 
$$ F''_{\lambda} (x) =  (\frac{x_3+t_1}{\delta_1})^{2k-2} e^{\lambda (\frac{x_3+t_1}{\delta_1})^2}((\frac{(2k)(2k-1)}{\delta_1^2}+\frac{8\lambda k +2\lambda}{\delta_1^2}(\frac{x_3+t_1}{\delta_1})^2+ \frac{4\lambda^2}{\delta_1^2}(\frac{x_3+t_1}{\delta_1})^4 )$$ for  $-t_1-\delta_1 \leq x_3 \leq -t_1$. \\
%\noindent We deduce the following estimate in region $C$ for $\lambda$ large enough:
 %\[
  % | F'_{\lambda} (x)| \leq \left\{\begin{array}{lr}
    %   C \lambda e^{\lambda (\frac{x_3-t_2}{\delta_2})^2}  & \text{for } t_2 \leq x_3 \leq t_2+\delta_2 \\
    %  C \lambda e^{\lambda (\frac{x_3+t_1}{\delta_1})^2}  , & \text{for } -t_1-\delta_1 \leq x_3 \leq -t_1\\
     %   \end{array}\right\}\\
%\\  
%\]
%\noindent Similarly we obtain:

 %\[
   % F''_{\lambda} (x) \geq \left\{\begin{array}{lr}
     %  C \lambda^2 e^{\lambda (\frac{x_3-t_2}{\delta_2})^2}  & \text{for } t_2 \leq x_3 \leq t_2+\delta_2 \\
     % C \lambda^2 e^{\lambda (\frac{x_3+t_1}{\delta_1})^2}  , & \text{for } -t_1-\delta_1 \leq x_3 \leq -t_1\\
      %  \end{array}\right\}\\
%\\  
%\]

\noindent Note that:

$$  D^2\phi_0(\nabla \phi_0,\nabla \phi_0) = (\partial_3 \phi_0)^2 \partial_{33} \phi_0 + 2\partial_1 \phi_0 \partial_3 \phi_0 \partial_{13}\phi_0.$$
So:
$$D^2\phi_0(\nabla \phi_0,\nabla \phi_0) \geq |\partial_3 \phi_0|(|\partial_3\phi_0| \partial_{33}\phi_0 - 2|\chi_0 \chi_0'|).$$
$$  |\partial_3\phi_0| = |x_1 \chi_0' +F'(x_3)| \geq |F'(x_3)| - |x_1| |\chi_0'|. $$
Using the Cauchy-Schwarz inequality we see that:
 \[
    |F'_{\lambda} (x)| \geq \left\{\begin{array}{lr}
        \frac{4}{\delta_2}\sqrt{ \lambda k}(\frac{x_3-t_2}{\delta_2})^{2k} & \text{for } t_2 \leq x_3 \leq t_2+\delta_2\\
     \frac{4}{\delta_2}\sqrt{ \lambda k}(\frac{x_3+t_1}{\delta_1})^{2k} & \text{for } -t_1-\delta_1 \leq x_3 \leq -t_1\\
        \end{array}\right\}\\
\\  
\]
And:
 \[
   |x_1| |\chi_0'| \leq \left\{\begin{array}{lr}
         C(\Omega) k^2  |(\frac{x_3-t_2}{\delta_2})|^{8k-1} & \text{for } t_2 \leq x_3 \leq t_2+\delta_2 \\
      C(\Omega) k^2  |(\frac{x_3+t_1}{\delta_1})|^{8k-1} , & \text{for }  -t_1-\delta_1 \leq x_3 \leq -t_1\\
        \end{array}\right\}\\
\\  
\]
\noindent Hence we can conclude that:
 \[
  |\partial_3\phi_0|  \geq \left\{\begin{array}{lr}
     \frac{4}{\delta_2}\sqrt{ \lambda k}(\frac{x_3-t_2}{\delta_2})^{2k} -  C(\Omega) k^2  |(\frac{x_3-t_2}{\delta_2})|^{8k-1}   & \text{for }  t_2 \leq x_3 \leq t_2+\delta_2 \\
   \frac{4}{\delta_2}\sqrt{ \lambda k}(\frac{x_3+t_1}{\delta_1})^{2k}-  C(\Omega) k^2  |(\frac{x_3+t_1}{\delta_1})|^{8k-1}   , & \text{for }   -t_1-\delta_1 \leq x_3 \leq -t_1\\
        \end{array}\right\}\\
\\  
\]
and therefore for $\lambda$ sufficiently large we obtain that:
 \[
  |\partial_3\phi_0|  \geq \left\{\begin{array}{lr}
     \frac{2}{\delta_2}\sqrt{ \lambda k}(\frac{x_3-t_2}{\delta_2})^{2k}   & \text{for }  t_2 \leq x_3 \leq t_2+\delta_2 \\
   \frac{2}{\delta_2}\sqrt{ \lambda k}(\frac{x_3+t_1}{\delta_1})^{2k}   , & \text{for }   -t_1-\delta_1 \leq x_3 \leq -t_1\\
        \end{array}\right\}\\
\\  
\]

\noindent Now:
$$  \partial_{33}\phi_0= x_1 \chi'_0 + F''(x_3) \geq \frac{1}{2}F''(x_3).$$
Hence:
 \[
   \partial_{33}\phi_0 \geq \left\{\begin{array}{lr}
      \frac{2\lambda}{\delta_2^2} (\frac{x_3-t_2}{\delta_2})^{2k}, & \text{for } t_2 \leq x_3 \leq t_2+\delta_2 \\
    \frac{2\lambda}{\delta_1^2} (\frac{x_3+t_1}{\delta_1})^{2k}, & \text{for }  -t_1-\delta_1 \leq x_3 \leq -t_1\\
        \end{array}\right\}\\
\\  
\]

\noindent Hence combining the above we see that for $\lambda$ sufficiently large we have that:
 $$ D^2\phi_0(\nabla \phi_0,\nabla \phi_0) \geq 0. $$

\noindent Let us now analyze the term  $ D^2\phi_0 (X,X) $ for all $X$ with $ \langle \nabla \phi_0, X \rangle =0$
\\
Note that $d\phi_0(X)=0$ implies that:
$$ X \in \spn \{ \partial_2, \partial_3 \phi_0 \partial_1 - \partial_1 \phi_0 \partial_3 \}.$$
but since $ g$ is Euclidean in this region we have the following:
$$ D^2\phi_0 (\partial_2,X)=0.$$ 
\\
Now:
$$D^2\phi_0 (\partial_3 \phi_0 \partial_1 - \partial_1 \phi_0 \partial_3,\partial_3 \phi_0 \partial_1 - \partial_1 \phi_0 \partial_3) = (\partial_1 \phi_0)^2 \partial_{33} \phi_0- 2\partial_1 \phi_0 \partial_3 \phi_0 \partial_{13}\phi_0.$$
So:
$$D^2\phi_0 (\partial_3 \phi_0 \partial_1 - \partial_1 \phi_0 \partial_3,\partial_3 \phi_0 \partial_1 - \partial_1 \phi_0 \partial_3) =\chi_0^2 (x_1 \chi''_0+ F'') - 2 \chi_0 \chi'_0(x_1 \chi'_0 + F').$$
Using the Cauchy-Schwarz inequality again and by looking at the sign of the $x_3$ we can get the following inequalities:
$$ -2 x_1 |\chi'_0|^2 \chi_0  -2 \chi_0 \chi'_0 F'\geq 0.$$
$$ F'' + x_1 \chi''_0 \geq  \frac{F''}{2} \geq 0.$$

\noindent and thus by combining the above inequalites we obtain that:
$$D^2\phi_0 (\partial_3 \phi_0 \partial_1 - \partial_1 \phi_0 \partial_3,\partial_3 \phi_0 \partial_1 - \partial_1 \phi_0 \partial_3) \geq 0.$$

\end{proof}

\noindent We will now provide a lemma that will show that the H\"{o}rmander Hypo-Ellipticity yields a global Carleman estimate in our manifold.

\begin{lemma}
Let $(\Omega,g)$ be a compact smooth Riemannian manifold with smooth boundary and suppose $\psi \in C^{2}(\Omega)$ is such that $d\psi \neq 0$ and the H\"{o}rmander Hypo-Ellipticity condition is satisfied:
\[ D^2\psi (X,X) + D^2\psi(\nabla \psi,\nabla\psi) \geq 0\]
whenever $|X|=|\nabla\psi|$ and $\langle \nabla\psi, X \rangle =0$. Let $\mathbb{N}:=\{x \in \partial \Omega : \partial_{\nu} \psi=0\}$ and let $\mathbb{W}:=\{v \in C^2(\Omega) : v|_{\partial \Omega}=0 , \partial_{\nu}v|_{\partial \Omega \setminus \mathbb{N}}=0 \}$.
Then there exists $C$ depending only on the domain and $h_0>0$ such that for all $v \in  \mathbb{W}$ and all $ 0<h<h_0$ the following estimate holds:
\[ \| e^{\frac{\psi}{h}} \triangle_g (e^{-\frac{\psi}{h}} v) \|_{L^2(\Omega)} \ge \frac{C}{h} \|v\|_{L^2(\Omega)} + C \|Dv\|_{L^2(\Omega)}  \]
\end{lemma}

\begin{remark}
The above estimate is called a {\bf Carleman estimate} in $(\Omega,g)$. The corresponding phase function $\psi$ is called a {\bf Carleman weight}. In general there is a rather standard technique of proving these estimates either through integration by parts or semiclassical calculus. We will employ the former method due to simplicity. In cases where \[ D^2\psi (X,X) + D^2\psi(\nabla \psi,\nabla\psi) >0 \]
whenever $|X|=|\nabla\psi|$ and $\langle \nabla\psi, X \rangle =0$ one can refer to \cite{EZ} for proving this estimate  where in fact we would get a stronger gain in terms of $h$. Similarly in the case where \[ D^2\psi (X,X) + D^2\psi(\nabla \psi,\nabla\psi) =0 \]
whenever $|X|=|\nabla\psi|$ and $\langle \nabla\psi, X \rangle =0$ one can refer to \cite{DKSU} or \cite{S} for a proof. In our setting we are in an intermediate case and thus require to adjust the arguments. \\
\end{remark}

\begin{proof}

It suffices to prove the claim for the renormalized metric $ \hat{g}=|\nabla^g \psi|_g^2 g $. To see this let us assume that $ c = |\nabla^g \psi|_g^{-2}$ and that $\psi$ is a Carleman weight with respect to $\hat{g}$. But then using the transformation property of 
Laplace Beltrami operator under conformal changes of metric we deduce that: 
\[ e^{\frac{\psi}{h}} (-h^2 \triangle_g) (e^{-\frac{\psi}{h}} v) =  e^{\frac{\psi}{h}} (-h^2 c^{-\frac{5}{4}} \triangle_{\hat{g}}) (c^{\frac{1}{4}}e^{-\frac{\psi}{h}} v) - h^2 q_c c^{-1} v\]
where:
\[q_c = c^{\frac{1}{4}}\triangle_{c\hat{g}} c^{-\frac{1}{4}}\]
Now note that $c(x)>0 $ for all $x \in \Omega$ and $ \|q_c\|_{L^{\infty}} < \infty$. Therefore :
\[ \| e^{\frac{\psi}{h}} (-h^2 \triangle_g) (e^{-\frac{\psi}{h}} v)\|_{L^2(g)} \gtrapprox h \|v\|_{L^2} +h^2\|Dv\|_{L^2} - h^2 \|q_c c^{-1}\|_{L^{\infty}} \|v\|_{L^2}\]
The claim will clearly follow for $h$ small enough.\\
\\

\noindent Let $P_{\psi} :=  e^{\frac{\psi}{h}} (-h^2\triangle_{\hat{g}}) e^{-\frac{\psi}{h}} = A +  B$ where $A$ and $B$ are the formally symmetric and anti-symmetric operators ( in $L^2(\Omega_1, \hat{g})$):\\

\[A= -h^2 \triangle_{\hat{g}} - 1\] 
\[B= h (2 \langle d\psi,d\cdot\rangle_{\hat{g}} + \triangle_{\hat{g}} \psi)\]\\

\noindent Hence:
\[\|P_{\psi}v\|^2_{L^2(\hat{g})}= \|Av\|^2_{L^2(\hat{g})}+\|Bv\|^2_{L^2(\hat{g})}+ ([A,B]v,v)_{L^2(\hat{g})}\]\\

\noindent Note that the key reason on why there will be {\bf no boundary terms} in the above expression is the assumption that $ v \in \mathbb{W}$. Now note that:
\[ [A,B] = -2h^3 [\triangle_{\hat{g}}, \langle d\psi,d\cdot\rangle_{\hat{g}}] + h^3 X\] 
where $X$ is a smooth vector field.

\noindent Let us define the coordinate system $(t,y_1,y_2)$ as follows:
Define the normal vector field to the level sets of $\psi$ and let the integral curves correspond to the coordinate $t$ choosing $t=0$ on one of these level sets. Furthermore let us consider smooth maps $G_t$ to be smooth diffeomorphisms from the unit disk to the corresponding level set $\psi_t$ smoothly depending on $t$. Note that in our coordinate system the pull back of the metric takes the following form : $$g= dt\otimes dt+ g_{\alpha \beta}(t,y) dy^{\alpha}\otimes dy^{\beta}$$

\noindent Thus:
$$([A,B]v,v)_{L^2(\hat{g})} = -2h^3 \int \partial_t \hat{g}^{\alpha\beta} \partial_{\alpha} v \partial_{\beta} v + h^3 \int K(x) |v|^2  $$

\noindent Here, $K$ deontes a continuous function on $\Omega$. We now note that  $-\partial_t \hat{g}^{\alpha\beta}$ denotes the inverse of the second fundamental form of the level sets of $\psi$ with respect to the renormalized metric. Recall that if  $ \Gamma^{n-1} \subset M^{n}$ is an embedded nondegenerate hypersurface in $M$, then the second funamental form $h(X,Y)$ on $\Gamma$ changes under conformal rescalings $\hat{g}=cg$ as follows:
\[   \hat{h}(X,X) = \sqrt{c}(h(X,X) + \frac{1}{2}\frac{\nabla_N c}{c} g(X,X))\] 
Hence:
\[ \hat{h}(X,X)=\sqrt{c}(D^2\psi(X,X) + D^2\psi(\nabla \psi,\nabla \psi) \frac{|X|^2}{|\nabla\psi|^2}) \]
Thus using the main assumption of the Lemma, we see that  $-\partial_t \hat{g}^{\alpha\beta}$ is positive semi-definite and thus we can conclude that:
\[\|P_{\psi}v\|^2_{L^2(\hat{g})} \geq \|Av\|^2_{L^2(\hat{g})}+ \|Bv\|^2_{L^2(\hat{g})}+ ([A,B]v,v)_{L^2(\hat{g})} \]
So:
\[\|P_{\psi}v\|^2_{L^2(\hat{g})}\geq \|Av\|^2_{L^2(\hat{g})}+  \|Bv\|^2_{L^2(\hat{g})} + h^3 \int K(x) |v|^2  \hspace{2mm} (*)\]
Note that:
\[Bv= h (2 \langle d\psi,dv\rangle_{\hat{g}} + (\triangle_{\hat{g}} \psi)v) = h ( 2\partial_t v + (\triangle_{\hat{g}} \psi) v)\]
The Poincare inequality implies that:
\[  \| \partial_t v\|_{L^2(\Omega,\hat{g})} \geq C \| v\|_{L^2(\Omega,\hat{g})} \hspace{1cm}  \forall v \in H^1_0 (\Omega)\]\\
Recall that the level sets of $\psi$ are non-trapping since $d\psi \neq 0$ anywhere. Since we are working over a compact manifold we can use an integrating factor and use the Poincare inequality above to conclude that:
\[ \|Bv\|_{L^2(\Omega,\hat{g})} \geq C h \| v\|_{L^2(\Omega,\hat{g})} \hspace{1cm}  \forall v \in C^{\infty}_c(\Omega)\hspace{1cm} (**) \]
Let us also observe that by integrating $Av$ against $\delta h^2v$ for some small $\delta$ independent of $h$ we obtain the following estimate:
\[ \|Av\|^2_{L^2(\hat{g})} \geq  C \delta ( h^4 \int |\nabla v|^2 - h^2 \int v^2) \hspace{1cm} (***)\] 

\noindent Combining (*),(**) and (***) yields the claim.

\end{proof}

\noindent Combining the previous two lemmas yields the following:

\begin{corollary}
Let $ \phi_0(x_1,x_2,x_3) = x_1 \chi_0(x_3) +(F_{\lambda}\circ \omega)(x)$ as defined in the previous lemma with  $k \geq 1$ arbitrary and $\lambda$ sufficiently large and only depending on the domain and k. Then $\phi_0(x_1,x_2,x_3) $ is a Carleman weight in $\Omega$, that is to say there exists $h_0>0$ and $C$ depending on the domain $(\Omega,g)$ such that the following estimate holds:

\[ \| e^{\frac{ \phi_0}{h}} \triangle_g (e^{-\frac{ \phi_0}{h}} v) \|_{L^2(\Omega)} \ge \frac{C}{h} \|v\|_{L^2(\Omega)} + C  \|Dv\|_{L^2(\Omega)}  \]
$\forall h \le h_0$ and $ v \in \mathbb{D}$.\\
\\

\end{corollary}

\section{Complex Geometric Optics}

\noindent In this section, we will utilize the above corollary to construct a family of complex geometric optic solutions (CGO) to the Schr\"{o}dinger equation $(-\triangle_g + q)u=0$ concentrating on the plane $\Pi$. Thes families of solutions can then be used to deduce uniqueness of the potential from the local Dirichlet to Neumann map $C_q^{\Gamma,\Gamma}$. We will closely follow the ideas in \cite{GU} and \cite{NS} .  \\

\begin{definition}
$$ P_{\tau} v := e^{-\tau \phi_0} (\triangle_g - q_*) ( e^{\tau \phi_0}v) $$
\end{definition}

\begin{definition}
$\pi_{\tau}:L^2(\Omega) \to L^2(\Omega)$ denotes the orthogonal projection onto:
$$\{ v \in L^2(\Omega): P_{\tau}v=0, v|_{\partial \Omega \setminus \Gamma}=0\}$$
\end{definition}

\begin{lemma}
\label{construction}
Let $ f \in L^2(\Omega,g)$. For all $\tau>0$ sufficiently large, there exists a unique function $r:=H_{\tau}f \in H_{\triangle}(\Omega)$ such that:
\begin{itemize}
\item{$P_{\tau}r =f$   }

\item{ $r|_{\partial \Omega \setminus \Gamma}=0$  }

\item{ $\pi_{\tau} r = 0$}

\end{itemize}

\noindent Furthermore $r$ satisfies the estimate:
\[   \|r\|_{L^2(\Omega)} \leq C \tau^{-1} \|f\|_{L^2(\Omega)}\]\\
where the constant $C$ only depends on $(\Omega,g)$ and $\|q_*\|_{L^{\infty}(\Omega)}$.
\end{lemma}

\begin{remark}
This is a rather standard proof about deducing surjectivity for some operator $T$ from the knowledge of injectivity and closed range for the adjoint operator $T^*$. We will closely follow the proofs provided in \cite{S} and \cite{NS}  here. \\
\end{remark}

\begin{proof}

Let us first prove uniqueness. Indeed suppose that $r_1,r_2$ are two solutions. Then $P_{\tau}(r_1-r_2)=0$ and $(r_1-r_2)|_{\partial \Omega \setminus \Gamma}=0$ so we have $\pi_{\tau}(r_1-r_2)=(r_1-r_2)$. However the last condtion in the lemma implies that $\pi_{\tau}(r_1-r_2)=0$ so $r_1 \equiv r_2$.
To show existence define $\mathbb{B}= P_{\tau}^* \mathbb{D}$ as a subspace of $L^2(\Omega)$ (recall Definition ~\ref{D}) . Consider the linear functional $L: \mathbb{B} \to \mathbb{C}$ through:
\[ L(P^*_{\tau} v) = \langle v,f \rangle \hspace{1cm} \forall v \in \mathbb{D}\]
This is well-defined since any element of $\mathbb{B}$ has a unique representation as $P^*_{\tau} v$  with $ v \in \mathbb{D}$ by the Carleman estimate. Also using the Cauchy-Schwarz inequality and the Carleman estimate we have:
\[ |  L(P^*_{\tau} v)| \leq \|v\|_{L^2} \|f\|_{L^2} \leq C \tau^{-1} \|f\|_{L^2} \|P^*_{\tau} v\|_{L^2}\]
for $\tau$ large enough with $C$ depending only on $(\Omega,g)$. Thus $L$ is a bounded linear operator on $\mathbb{B}$. Extend $L$ by continuity to the closure $\bar{\mathbb{B}}$. Set $L \equiv 0$ on the orthogonal complement in $L^2(\Omega)$ of $\mathbb{B}$. Thus we obtain a bounded linear operator $\hat{L}: L^2(\Omega) \to \mathbb{C}$ with $ \hat{L} |_{\mathbb{D}} = L$. Furthermore:
\[ \| \hat{L}\| \leq C \tau^{-1} \|f\|_{L^2}\]
Now by the Riesz representation therorem we deduce that there exists a unique $r \in  L^2(\Omega)$ such that $\hat{L}(w) = \langle w,r\rangle \hspace{5mm}\forall w \in L^2(\Omega) $ and $(1-\pi_{\tau})r=r$. we also have  $\| r\|_{L^2} \leq C \tau^{-1} \|f\|_{L^2}$. Note that for $w \in C^{\infty}_c(\Omega)$ we have:\\
 \[ \langle v,P_{\tau}r \rangle =   \langle P^*_{\tau}v,r\rangle = \hat{L}(P^*_{\tau}v)= L(P^*_{\tau}v)=\langle v,f \rangle \]
Hence $ P_{\tau}r=f$ in the weak sense. To show that $r|_{\partial \Omega \setminus \Gamma}=0$ we note that for any $ v \in \mathbb{D}$:
$$\langle  P_{-\tau}v,r \rangle = \langle v ,f  \rangle$$
Using the Green's identity we know that:
$$\langle   P_{-\tau}v,r \rangle = \langle  v,P_{\tau} r \rangle + \int_{\partial \Omega \setminus \Gamma} (\partial_{\nu} v)r $$ 
Combining these we get the result.
\end{proof}

\noindent With the proof of Lemma ~\ref{construction} now complete, one can proceed with construction of the CGO solutions as follows. 
\noindent Let us define the function $\Phi:\Omega \to \mathbb{C}$ through $\Phi = \phi_0 + i \tilde{\omega}$. We also define $v_0:U \to \mathbb{R}$ through $v_0 = h(x_1+ix_2) \chi(x_3)$ where $h$ is an arbitrary holomorphic function in $z:=x_1+ix_2$ and $\chi$ is an arbitrary function of compact support in the set $V$. Note that in the region $V$ we have the following equations (recall that the metric $g$ is Euclidean in this region):
$$ \langle d\Phi,d\Phi\rangle_g =0$$
$$ 2\langle d\Phi,dv_0 \rangle_g + (\triangle_g \Phi) v_0 =0$$

\noindent Subsequently, we have the following two Lemmas:

\begin{lemma}
\label{harmonic1}
For $\tau>0$ sufficiently large, there exists solutions $u_0$ of $(-\triangle_g +q_*) u_0=0$ of the form $ u_0 = e^{\tau \Phi} (v_0 + r_0)$ where $r_0|_{\partial \Omega \setminus \Gamma}=0$ and $\|r_0\|_{L^2(\Omega)} \leq \frac{C}{\tau}$. Here $C$ is a constant that depends on the domain $(\Omega,g)$ and $\|q_*\|_{L^{\infty}(\Omega)}$.
\end{lemma}

\begin{proof}

Let us first consider solving the equation $$ P_{\tau} r =  e^{-\tau \phi_0} (\triangle_g-q_*) (e^{\tau\phi_0} r)  =-e^{-\tau (\phi_0 - \Phi )   } e^{-\tau \Phi} (\triangle_g-q_*) (e^{\tau\Phi} v_{0})$$
Since $v_{0}$ is compactly supported in the region $V$:
$$e^{-\tau \Phi} (\triangle_g-q_*) (e^{\tau \Phi} v_{0})=  \tau^2 \langle d\Phi,d\Phi\rangle_g v_0 + \tau [ 2\langle d\Phi,dv_0 \rangle_g + (\triangle_g \Phi) v_0 ] + \triangle_g v_0-q_*v_0$$
Hence using the construction formulas for $\Phi$ and $v_0$ and noting that $\phi_0 - \Phi$ is purely imaginary, we can immediately conclude that $ \|  e^{-\tau \Phi} (\triangle_g-q_*) e^{\tau\Phi} v_{0} \|_{L^2(\Omega)} \leq C$ for some constant $C$. 
This is simply due to the fact that in $V$ we have the following:
$$ \langle d\Phi,d\Phi\rangle_g=0$$
$$2\langle d\Phi,dv_0 \rangle_g + (\triangle_g \Phi) v_0=0$$

\noindent  Let $$r= -H_{\tau}(e^{-\tau (\phi_0 - \Phi  )   } e^{-\tau \Phi} (\triangle_g-q_*) (e^{\tau\Phi} v_0))$$
 We can now choose $ r_{0} = e^{\tau (\phi_0 - \Phi)} r$ to conclude the proof.\\

\end{proof}

\begin{lemma}
\label{schrodinger1}
Let $q_1 \in L^{\infty} (\Omega)$. For $\tau>0$ sufficiently large, there exists solutions $u_1$ of $(-\triangle_g + q_1)u_1=0$ of the form $ u_1 = e^{\tau \Phi} (v_0 + r_1)$ where $r_1|_{\partial \Omega \setminus \Gamma}=0$ and  $\|r_1\|_{L^2(\Omega)} \leq \frac{C}{\tau}$.
\end{lemma}

\begin{proof}
Consider the equation:
 $$ - e^{-\tau \phi_0} (\triangle_g-q_1) (e^{\tau\phi_0} r) + q r =e^{-\tau (\phi_0 - \Phi  )   } e^{-\tau \Phi} (-\triangle_g+q_1) (e^{\tau \Phi} v_0)=:f $$
but since $v_{0}$ is compactly supported in $V$:
$$e^{-\tau \Phi} \triangle_g (e^{\tau \Phi} v_{0})=  \tau^2 \langle d\Phi,d\Phi \rangle_g v_0 + \tau [ 2\langle d\Phi,dv_0 \rangle_g + (\triangle_g \Phi) v_0 ] + (\triangle_g-q_1) v_0$$
Recall that in the region $V$ we have the following:
$$ \langle d\Phi,d\Phi\rangle_g=0$$
$$2\langle d\Phi,dv_0 \rangle_g + (\triangle_g \Phi) v_0=0$$ 
Hence, we can immediately conclude that $ \|  e^{-\tau \Phi} \triangle_g (e^{\tau \Phi} v_{0}) -q_1 v_0\|_{L^2(\Omega)} \leq C$ for some constant $C$.

\noindent Motivated by Lemma ~\ref{construction} we try the ansatz $r = H_{\tau} \tilde{r}$ to obtain:
 $$(- I + (q_1-q_*)H_{\tau}) \tilde{r} =f$$
But $H_\tau: L^2(\Omega) \to L^2(\Omega)$ is a contraction mapping for $\tau$ large enough with $\|H_{\tau}\| \leq \frac{C}{\tau}$ and thus for sufficiently large $\tau$ the inverse map $(I+(q_1-q_*)H_{\tau})^{-1}:L^2(\Omega) \to L^2(\Omega)$ exists and it is given by the following infinite Neumann series:
$$ (-I + (q_1-q_*) H_{\tau})^{-1} = -\sum_{j=0}^{\infty} ((q_1-q_*)H_{\tau})^{j}$$

\noindent Hence:

$$ \|(I + (q_1-q_*) H_{\tau})^{-1}\|_{L^2(\Omega) \to L^2(\Omega)}\leq C$$

\noindent So we deduce that if :
$$ r = H_\tau ( I +(q_1-q_*)H_{\tau})^{-1} f$$ 
then if we choose $ r_1 = e^{\tau (\phi_0 - \Phi)} r$  we have that $ u_1 = e^{\tau \Phi} (v_{0} + r_1)$ solves $(-\triangle_g + q_1)u_1=0$ and furthermore:
$$\|r_1\|_{L^2(\Omega)} \leq \frac{C}{\tau}$$
\\
\end{proof}

\noindent Let $\psi_0(x)= -x_1 \chi_0(x) + (F_{\lambda}\circ\omega)(x)$. Note that we have the following estimate as a result of Lemma ~\ref{hormander}:

\[ \| e^{\frac{ \psi_0}{h}} \triangle_g (e^{-\frac{ \psi_0}{h}} v) \|_{L^2(\Omega)} \ge \frac{C}{h} \|v\|_{L^2(\Omega)} + C  \|Dv\|_{L^2(\Omega)}  \]
$\forall h \le h_0$ and $ v \in \mathbb{D}$.\\
\\

\begin{definition}
$$ Q_{\tau} v = e^{-\tau \psi_0} (\triangle_g-q_*) ( e^{\tau \psi_0}v) $$
\end{definition}

\begin{definition}
$\tilde{\pi}_{\tau}$ denotes the orthogonal projection onto:
$$\{ v \in L^2(\Omega): Q_{\tau}v=0, v|_{\partial \Omega \setminus \Gamma}=0\}$$
\end{definition}

\noindent Thus we can state the following Lemma which is a direct parallel to Lemma ~\ref{construction}:

\begin{lemma}
Let $ f \in L^2(\Omega,g)$. There exists a unique function $r:=L_{\tau}f \in H_{\triangle}(\Omega)$ such that:
\begin{itemize}
\item{$Q_{\tau}r =f$   }

\item{ $r|_{\partial \Omega \setminus \Gamma}=0$  }

\item{ $\tilde{\pi}_{\tau} r = 0$}

\end{itemize}

\noindent Furthermore for $\tau$ large enough, $r$ satisfies the estimate:
\[   \|r\|_{L^2(\Omega)} \leq C \tau^{-1} \|f\|_{L^2(\Omega)}\]\\
where the constant $C$ only depends on $\Omega$.
\end{lemma}

\noindent Let $\Psi = \psi_0 - i \tilde{\omega}$. Notice that for $x \in V $ we have that $ \Psi= -x_1 - i x_2$ and that $\Re(\Psi)= \psi_0$. Finally we note that for $x \in V$ we have:
$$  \langle d\Psi, d\Psi \rangle_g  = 0$$ and:
$$ 2  \langle d\Psi, dv_0 \rangle_g + (\triangle_g \Psi)v_0 = 0$$
\\
\noindent Thus we can state the following corollary to Lemma ~\ref{schrodinger1}:

\begin{corollary}
Let $q_2 \in L^{\infty} (\Omega)$. For all $\tau>0$ sufficiently large, there exists solutions $u_2$ to $(-\triangle_g + q_2)u_2=0$ of the form $ u_2 = e^{\tau \Psi} (v_0 + r_2)$
where $\|r_2\|_{L^2(\Omega)} \leq \frac{C}{\tau}$.\\
\end{corollary}

\section{Proof Of Uniqueness}

\begin{proof}[Proof of Theorem ~\ref{partial}]

Suppose $q_1,q_2 \in C(\overline{\Omega})$ satisfy $C_{q_1}^{\Gamma,\Gamma} = C^{\Gamma,\Gamma}_{q_2}$.  Let us use the Green's identity to $ u_1 = e^{\tau \Phi} (v_0 + r_1)$ and $u_2=e^{\tau \Psi}(v_0+r_2)$. Thus:
$$  I_{\tau} = \int_{\partial\Omega} u_2 \partial_{\nu} u_1  - \int_{\partial\Omega} u_1 \partial_{\nu} u_2 = \int_{\Omega} u_2 \triangle_g u_1 - \int_{\Omega} u_1 \triangle_g u_2 $$
Let $q:=q_1-q_2$. Since $q_1|_{V^c}=q_2|_{V^c}=q_*|_{V^c}$ we have:
$$ I_{\tau}= \int_{V} q u_1 u_2.  $$

\noindent Note that since $C_{q_1}^{\Gamma,\Gamma} = C^{\Gamma,\Gamma}_{q_2}$ and since $u_1|_{\partial \Omega \setminus \Gamma} = u_2|_{\partial \Omega \setminus \Gamma}=0$ we have that:
$$ I_{\tau} =0 $$
\noindent So:
$$0= \int_{V} q u_1 u_2 = \int_{V} q (v_0+r_1)(v_0+r_2).$$
\noindent Using the Cauchy-Schwarz inequality we see that:
$$|\int_{V} q r_1 r_2| \leq \frac{C}{\tau^2}$$
$$|\int_{V} q r_1 v_0| \leq \frac{C}{\tau}$$
$$|\int_{V} q r_2 v_0| \leq \frac{C}{\tau}.$$
\noindent Thus by taking the limit as $\tau \to \infty$ we obtain:

$$ 0 = \int_{V} q v_0^2. $$ 

\noindent Recall that $v_0(x)= h(z) \chi(x_3)$. Thus we see that by choosing $\chi(x_3)$ approximating a delta distribution we have the following:
$$\int_{\Pi} q h(z) \equiv 0. $$

\noindent In particular this implies that given any plane in the convex hull of $\Gamma$, intergrals of the function $q$ on the plane vanishes. At this point, one can use the local injectivity of the Radon transform for continuous functions of compact support (see for example \cite{H}) to conclude that:
 $$q|_U \equiv 0.$$
\end{proof}

\bigskip
\noindent The methods presented in this paper are quite robust. Let us now state a few remarks about possible generalizations of Theorem ~\ref{partial}:

\begin{remark}
The results in this paper can easily be generalized to higher dimensions, $n\geq 3$. One can indeed produce similar CGO solutions concentrating on two-planes for any $n \geq 3$ and use \cite{H} to obtain uniqueness of the potential. 
\end{remark}

\begin{remark}
The proof presented here is not constructive. One can give a reconstruction algorithm for the CGO solutions at the boundary from the local Dirichlet to Neumann map through the approach that we introduced in \cite{Feiz}. The method in that paper uses an artificial extension of the manifold that produces a boundary integral equation. That approach can be adapted here without much difficulty. The key difference would be the need for a new proof of Lemma 8.5 as we are working in less regular Sobolev spaces here.
\end{remark}

\begin{remark}

In the spirit of the results obtained in \cite{Feiz}, one can generalize the results in this paper to the setting where $\Gamma$ is not connected. Similarly one can generalize the results to the case where $U$ is conformally transversally anisotropic.

\end{remark}

\begin{remark}
It may be possible to adjust the arguments slightly to provide a logarithmic stability estimate for the partial data problem as well.
\end{remark}

%\begin{remark}
%It is possible to adjust the arguments to obtain a stability estimate on the local Dirichlet to Neumann map as well. This would rely on having an stability estimate on the local geodesic ray transform.
%\end{remark}

%\newpage
%\section{Appendix}

% Every LaTeX document must end with \end{document}.

\end{document}